\documentclass[a4paper]{amsart}
\usepackage[utf8]{inputenc}
\usepackage{amsmath, amscd, amsthm, amsfonts, amssymb}    
\usepackage{mathrsfs, bbm, cancel, stmaryrd}
\usepackage{physics}
\usepackage{hyperref}
\usepackage{textcomp}
\usepackage{comment}
\usepackage{scalerel,stackengine}
\usepackage{geometry}
\usepackage{mathtools}
\usepackage{dsfont}
\usepackage{musicography}
\usepackage{esint}
\usepackage{tikz-cd}
\usepackage{cancel}
\numberwithin{equation}{subsection}
\usepackage{enumitem}
\usepackage{mathrsfs} 
\usepackage{enumitem}
\usepackage{graphics} 
\usepackage[margin=10pt,font=small,labelfont=bf]{caption}
\usepackage{chapterfolder}
\usepackage{multicol}

\newcommand{\RR}{\mathbb{R}}

\newcommand{\NN}{\mathbb{N}}

\newcommand{\frg}{\frak{g}}
\newcommand{\frh}{\frak{h}}

\newcommand{\Id}{\mathrm{id}}

\newcommand{\ve}{\varepsilon}

\newcommand{\cyc}[3]{\sum_{\mathrm{cyc}\,(#1,#2,#3)}}

\newcommand{\longto}{\longrightarrow}
\newcommand{\longsto}{\longmapsto}
\newcommand{\evf}{\mathscr{E}}

\newcommand{\Rnabla}{R^\nabla}

\newcommand\equalhat{\mathrel{\stackon[1.5pt]{=}{\stretchto{%
    \scalerel*[\widthof{=}]{\wedge}{\rule{1ex}{3ex}}}{0.5ex}}}}
\makeatletter
\newcommand{\extp}{\@ifnextchar^\@extp{\@extp^{\,}}}
\def\@extp^#1{\mathop{\bigwedge\nolimits^{\!#1}}}
\makeatother

\newtheorem{theorem}{Theorem}
\newtheorem{proposition}{Proposition}[section]
\newtheorem{definition}{Definition}[section]
\newtheorem{lemma}{Lemma}[section]
\newtheorem{corollary}{Corollary}[section]

\newtheorem*{theorem1}{Theorem 1}
\newtheorem*{theorem2}{Theorem 2}

\DeclareMathOperator{\Sym}{Sym}
\DeclareMathOperator{\Hol}{Hol}
\DeclareMathOperator{\hol}{\frak{hol}}

\DeclareMathOperator{\GL}{GL}
\DeclareMathOperator{\End}{End}

\DeclareMathOperator{\perm}{perm}

\DeclareMathOperator{\gen}{span}
\DeclareMathOperator{\ad}{ad}
\DeclareMathOperator{\Ad}{Ad}

\title{Torsion-free Connections with prescribed Curvature}
\author{Efrain Basurto-Arzate}
\address{Fakult\"at f\"ur Mathematik, Technische Universit\"at Dortmund, Dortmund, Germany}
\email{efrain.basurto@math.tu-dortmund.de}
\date{}

\begin{document}

\maketitle

\begin{abstract}
 We provide necessary and sufficient conditions for a curvature map $S\colon U\longto K(\frg)$ to arise from the curvature tensor of a torsion-free connection on a sufficiently small $U'\subset U$ by using a suitable power series approach. This torsion-free connection is uniquely given. The curvature map $S$ is used to effectively compute the holonomy of this connection.
\end{abstract}

\tableofcontents

%§§§§§§§§§§§§§§§§§§§§§§§§§§§§§§§§§§§§§§§§§§§§§§§§§§§§§§§§§§§§§§§§§§§§§§§§§§§§§§§§§§§§§§§§§§§
\section{Introduction}\label{Sec:Intro}

Since its genesis in the middle of the 19th century, the concept of a connection has grown to become one of the classical objects of study in Differential Geometry. In its most elementary interpretation, a connection provides us with a way to differentiate sections of a vector bundle. Formally, a connection on the vector bundle $E\longto M$ is a linear map $\nabla\colon\Gamma(E)\longto\Gamma(T^*M\otimes E)$ that satisfies the \emph{Leibniz rule}:
\[
 \nabla(fs)=\dd f\otimes s+f\nabla s,
\]
for all $s\in \Gamma(E)$, $f\in C^\infty(M)$.

Associated to the connection $\nabla$ is its curvature tensor, which is defined as the unique smooth tensor field  $R^\nabla\in\Gamma(\extp^2T^*M\otimes\End(E))$ given by
\[
 R^\nabla(X,Y)s\coloneqq R^\nabla (X\wedge Y\otimes s)=\nabla_X\nabla_Ys-\nabla_Y\nabla_Xs-\nabla_{[X,Y]}s,
\]
for every $X,Y\in\frak{X}(M)$, $s\in\Gamma(E)$.

In the particular case $E=TM$, a connection naturally defines another important tensor field, the so-called torsion tensor, which is defined as the tensor field $T^\nabla\in\Gamma(\extp^2T^*M\otimes TM)$ given by
\[
 T^\nabla(X,Y)\coloneqq\nabla_XY-\nabla_YX-[X,Y].
\]

The connection is said to be \emph{torsion-free} if $T^\nabla\equiv 0$.

The curvature tensor of a torsion-free connection satisfies the \emph{Bianchi identities}:

\begin{subequations} 
\begin{align} 
\cyc{X}{Y}{Z}R^\nabla(X,Y)Z\coloneqq& R^\nabla(X,Y)Z+R^\nabla(Y,Z)X+R^\nabla(Z,X)Y=0, \label{Eq:1Bianchi}\tag{B1}\\ 
\cyc XYZ(\nabla_XR^\nabla)(Y,Z)=&(\nabla_XR^\nabla)(Y,Z)+(\nabla_YR^\nabla)(Z,X)+(\nabla_ZR^\nabla)(X,Y)=0\label{Eq:2Bianchi}\tag{B2}
\end{align} 
\end{subequations}

The Bianchi identities can be algebraically encapsulated. Indeed, given a finite vector space $V$, for a Lie subalgebra $\frh\subset\frg\coloneqq\frak{gl}(V)$ we define the space of \emph{algebraic curvature tensors}, and that of \emph{algebraic curvature derivatives} as the the subspaces
\begin{align*}
 K(\frh)\coloneqq\qty{R\in\extp^2V^*\otimes\frh\;\left\vert\;\cyc xyz R(x,y)z=0\in V\,\forall x,y,z\in V\right.},\\
 K^1(\frh)\coloneqq\qty{\phi\in V^*\otimes K(\frh)\;\left\vert\;\cyc xyz \phi(x)(y,z)=0\in\frh\;\forall x,y,z\in V\right.}.
\end{align*}

A consequence of the Ambrose-Singer holonomy Theorem (see Proposition \ref{Prop:Ambrose-Singer}) is the fact that for a torsion-free connection $\nabla$ on a smooth, connected manifold $M$, its curvature tensor $R^\nabla$ satisfies, for every $x\in M$:
\begin{align*}
 R_x^\nabla\in& K(\hol_x(\nabla)),\\
 (\nabla R^\nabla)_x\in& K^1(\hol_x(\nabla)).
\end{align*}

These facts esentially are the motivation for the main question tackled in this article:

\begin{center}
 \emph{Under which conditions can it be guaranteed, that a given curvature map $S\colon U\subset V\longto K(\frak{g})$ can be induced by the curvature tensor of a torsion-free connection?}
\end{center}

The Main Theorem of the article provides an answer to the previous question:

\begin{theorem1}\label{Thm:MainThm}
  Let $U\subset V$ a star-shaped around $0$ open subset, let $S\colon U\longto K(\frak{g})$ be a real-analytic map, $\theta$ the analytic solution to the singular initial value problem
 \[
  \begin{cases}
   \mathcal{L}_\evf(\mathcal{L}_\evf-\Id)\theta=S(\evf,\theta)\evf,\\
   \theta_0=\Id,\,\dd\theta_0=0,
  \end{cases}
 \]
and suppose $\omega\coloneqq I(S(\evf,\theta))$ satisfies the consistency relation $\dd\omega+\omega\wedge\omega=S(\theta,\theta)$. Then there exists a unique torsion-free analytic connection $\nabla$ on a sufficiently small open neighborhood $U'\subset U$ of $0$ such that for all $v\in U'$:
  \begin{equation}\label{Eq:S=PRIntro}
  \tag{\ref{Eq:S=PR}}
   S_v=P_{1,0}\cdot \Rnabla_{\gamma_v(1)},
  \end{equation}
  where $P_{1,0}$ denotes the inverse of the parallel translation map along the radial geodesic $\gamma_v(t)\coloneqq tv$, and $\Rnabla$ denotes the curvature tensor of the connection $\nabla$.
 \end{theorem1}
 
 In the statement of Theorem \ref{Thm:MainThm}, $\mathcal{L}_\evf$ denotes the Lie derivative along the the Euler vector field $\evf$, whereas $I$ denotes the \emph{integration map} on vector-valued forms on $U$ introduced in Section \ref{Sec:NecConditions} (see \eqref{Eq:ExplicitI}).
 
 As it turns out, Theorem \ref{Thm:MainThm} also provides us with a way to effectively compute the holonomy of torsion-free connections. Indeed, we have the following
 
 \begin{theorem2}\label{Thm:Thm2}
  Let $S\colon U\longto K(\frg)$, $\theta\in\Omega^1(U,V)$, $\omega\in\Omega^1(U,\frg)$ as in Theorem \ref{Thm:MainThm}. Let $\nabla$ be the torsion-free connection on $TU'$ which satisfies \eqref{Eq:S=PR}. It holds:
  \begin{equation}\label{Eq:holOfNablaSIntro}
  \tag{\ref{Eq:holOfNablaS}}
   \hol_0(\nabla)=\gen\qty{S_v(x,y)\;\vert\;v\in U',\,x,y\in V}.
  \end{equation}
 \end{theorem2}
 
 In other words, Theorem \ref{Thm:Thm2} states, that to compute the holonomy of torsion-free connections, it is enough to consider the action of the parallel translation map along radial geodesics, provided the initial curvature map is amenable enough.
 
 The article is structured as follows: Section \ref{Sec:Prelim} presents the required tools for the proper statement and proof of the results. In \ref{Subsec:CurvAndTor} some elementary properties of the curvature tensor of torsion-free connections are discussed, and how all of them are to be interpreted with respect to the invariant counterparts of these objects in the principal bundle setting. In \ref{Subsec:HolonomyPrelim} some basic facts and classical results on holonomy theory are collected. \ref{Subsec:ExpFraming} introduces the notion of an \emph{exponential framing}, a particular local trivialization given by parallel translation along radial geodesics which turns out to be the adequate interpretation in tackling the main problem.
 
Section \ref{Sec:NecConditions} provides us with the necessary conditions a curvature map needs to satisfy for it to arise from the curvature tensor of a torsion-free connection. In \ref{Subsec:PrincBundlePicture} we look in deeper detail into the invariant counterparts of torsion-free connections and its curvature tensor and how it naturally leads to a relevant consistency relation between the curvature form and an appropriate curvature map (see \eqref{Eq:F-RThetaTheta}, \eqref{Eq:Pulled-back F-RThetaTheta}). \ref{Subsec:AnIVP} is devoted to show that the tautological form, when pull-backed by an exponential framing, is a solution to the singular initial value problem stated in Proposition \ref{Prop:IVPThetaHat}. 

Section \ref{Sec:MainResult} is almost exclusively dedicated to the proof of Theorem \ref{Thm:MainThm}. It begins by showing Proposition \ref{Prop:MainIVP}, which deals with the existence and uniqueness of analytic solutions for the singular initial value problem introduced in \ref{Subsec:AnIVP}. Later in Proposition \ref{Prop:ThetaProperties} it is showed that this analytic solution behaves like a pulled-back tautological form, which ultimately ends up being the case, as proved in Theorem \ref{Thm:MainThm}.

Section \ref{Sec:Holonomy} entirely consists of the proof of Theorem \ref{Thm:Thm2}, as well as of Corollary \ref{Cor:HolonomyInSubalg}.

\section*{Acknowledgment} The author wishes to express his gratitude to L. Schwachhöfer for the valuable discussions during the preparation of the article, as well as for his comments on previous versions of this paper, which greatly helped to improve the presentation.

%§§§§§§§§§§§§§§§§§§§§§§§§§§§§§§§§§§§§§§§§§§§§§§§§§§§§§§§§§§§§§§§§§§§§§§§§§§§§§§§§§§§§§§§§§§§
\section{Preliminaries}\label{Sec:Prelim}

\subsection{The Curvature Tensor of an Affine Connection}\label{Subsec:CurvAndTor}

Throughout this article, unless otherwise stated, $M$ is a smooth, connected, and simply-connected $n$-manifold, $\nabla$ a connection on $TM$, $V=\RR^n$ equipped with its standard inner product, $G$ the Lie group $\GL(V)$, whereas its Lie algebra will be denoted by $\frg$.

In this setting, we define the \emph{torsion} and the \emph{curvature} of the connection $\nabla$ as the tensor fields $T^\nabla\in\Gamma\qty(\extp^2T^*M\otimes TM)$, $R^\nabla\in\Gamma\qty(\extp^2T^*M\otimes\End(TM))$ defined by
\begin{align}
 T^\nabla(X,Y)&=\nabla_XY-\nabla_YX-[X,Y],\label{Eq:TorsionTensor}\\
 R^\nabla(X,Y)&=\nabla_X\nabla_Y-\nabla_Y\nabla_X-\nabla_{[X,Y]},\label{Eq:CurvatureTensor}
\end{align}
for all vector fields $X,Y\in\frak{X}(M)$.

In the language of principal bundles, let $(F(M),\pi,M;G)$ denote the frame bundle of the manifold $M$. That is, the $G$-principal bundle, where each fiber $F(M)_x$ consists of the linear isomorphisms $u\colon V\longto T_xM$. The connection $\nabla$ uniquely defines a connection form\\ $\omega^\nabla\in\Omega^1(F(M),\frg)$. This form is characterized by two conditions:
\begin{itemize}[noitemsep]
 \item[i)] For every $g\in G$, $R_g^*\omega^\nabla=\Ad(g^{-1})\circ\omega^\nabla$, where $R_g\colon F(M)\longto F(M)$ denotes the action of $G$ on $F(M)$, that is, $R_g(u)=ug\coloneqq u\circ g$.\\
 \item[ii)] For every $X\in\frg$, $\omega^\nabla(\widetilde{X})=X$, where $\widetilde{X}$ denotes the vector field on $F(M)$ given by\\ $\widetilde{X}_u\coloneqq\eval{\dv{}{t}}_0ue^{tX}$.
\end{itemize}

These conditions imply that $u\longto \ker\omega^\nabla_u$ defines a right-invariant distribution on $F(M)$, which is called the \emph{horizontal distribution} of $\nabla$. Relevant is the fact that any vector field\\ $X\in\frak{X}(M)$ admits a unique \emph{horizontal lift}. That is, a vector field $\overline{X}\in\Gamma(\ker\omega^\nabla)$ $\pi$-related to $X$.

The \emph{tautological form} of the connection is defined as the form $\theta^\nabla\in\Omega^1(F(M),V)$ given by $\theta^\nabla_u(X)\coloneqq u^{-1}(\dd\pi(X))$, for $u\in F(M)_{\pi(u)}$, $X\in T_uF(M)$.

In terms of the connection and tautological forms one defines the \emph{torsion form} $\Theta^\nabla\in\Omega^2(F(M),V)$, and the \emph{curvature form} $F^\nabla\in\Omega^2(F(M),\frg)$ by the formulas:
\begin{align}
 \Theta^\nabla\coloneqq& \dd\theta^\nabla+\omega^\nabla\wedge\theta^\nabla,\label{Eq:TorsionForm}\\
 F^\nabla\coloneqq&\dd\omega^\nabla+\omega^\nabla\wedge\omega^\nabla.\label{Eq:CurvatureForm}
\end{align}

The fundamental link between the torsion and curvature forms with their tensor field counterparts is given through Lie differentiation. Concretely, for $k\in\NN_0$, $x\in M$, $u\in F(M)_x$, $X_1,\ldots, X_k,Y,Z\in\frak{X}(M)$ (\cite[Chapter~III, Theorem~5.1]{kobayashi1963foundations}):
\begin{align}
 ((\nabla^k_{X_1,\ldots,X_k}T^\nabla)(Y,Z))_x=&u(\mathcal{L}_{\overline{X}_1}\cdots\mathcal{L}_{\overline{X}_k}\Theta^\nabla(\overline{Y},\overline{Z}))_u,\label{Eq:T=uTheta}\\
 ((\nabla^k_{X_1,\ldots,X_k}R^\nabla)(Y,Z))_x=&u(\mathcal{L}_{\overline{X}_1}\cdots\mathcal{L}_{\overline{X}_k}F^\nabla(\overline{Y},\overline{Z}))_uu^{-1}.\label{Eq:R=AduF}
\end{align}

We say that the connection $\nabla$ is \emph{torsion-free} if its torsion identically vanishes. For the remaining of the work we assume the connection $\nabla$ to be torsion-free. In this situation the curvature tensor satisfies the \emph{first} and \emph{second Bianchi identities} \eqref{Eq:1Bianchi}, \eqref{Eq:2Bianchi}. That is, for every vector fields $X,Y,Z$,

\begin{subequations} 
\begin{align*} 
\cyc{X}{Y}{Z}R^\nabla(X,Y)Z=& 0,\\  
\cyc XYZ(\nabla_XR^\nabla)(Y,Z)=&0.
\end{align*} 
\end{subequations}

In terms of the connection and tautological forms, the first and second Bianchi identities respectively read:
\begin{subequations}
\begin{align}
 F^\nabla\wedge\theta^\nabla=&0,\label{Eq:1BIForm}\\
 \dd F^\nabla+\ad(\omega^\nabla)\wedge F^\nabla=&0\label{Eq:2BIForm}
\end{align}
\end{subequations}

Motivated by the Bianchi identities satisfied by the curvature tensor of the torsion-free connection $\nabla$ we make the following

\begin{definition}
 Let $\frak{h}\subset\frak{g}$ be a Lie subalgebra. We define for $m\in\NN_0$ the \emph{space of $m$-th order algebraic curvature derivatives} as the linear subspace of $\Sym^mV^*\otimes\extp^2V^*\otimes\frak{h}$ defined by 
 \begin{equation}\label{Eq:KmOfh}
  \begin{split}
   K^{(m)}(\frak{h})=&\ker\qty{\Sym^mV^*\otimes\extp^2V^*\otimes\frak{h}\longto\Sym^mV^*\otimes\extp^3V^*\otimes V}\\
    &\cap\ker\qty{\Sym^mV^*\otimes\extp^2V^*\otimes\frak{h}\longto\Sym^{m-1}V^*\otimes\extp^3V^*\otimes \frak{h}},
  \end{split}
 \end{equation}
 where the maps above are given by the composition of the natural maps
 \[
  \begin{tikzcd}
   \Sym^mV^*\otimes\extp^2V^*\otimes\frak{h}\arrow[r, hook]&\Sym^mV^*\otimes\extp^2V^*\otimes V^*\otimes V\arrow[r]&\Sym^mV^*\otimes\extp^3V^*\otimes V,
  \end{tikzcd}
  \]
  and
 \[
  \begin{tikzcd}
   \Sym^mV^*\otimes\extp^2V^*\otimes\frak{h}\arrow[r]&\Sym^{m-1}V^*\otimes V^*\otimes\extp^2V^*\otimes\frak{h}\arrow[r]&\Sym^{m-1}V^*\otimes\extp^3V^*\otimes\frak{h}.
  \end{tikzcd}
\]
\end{definition}

For $m=0,1$ one has the explicit descriptions
\[
 K(\frak{h})\coloneqq K^{(0)}(\frak{h})=\qty{\left.R\in\extp^2V^*\otimes\frak{h}\;\right\vert\;R(x,y)z+R(y,z)x+R(z,x)y=0,\,x,y,z\in V},
\]
\[
 K^1(\frak{h})\coloneqq K^{(1)}(\frak{h})=\qty{\left.\phi\in V^*\otimes K(\frak{h})\;\right\vert\;\phi(x)(y,z)+\phi(y)(z,x)+\phi(z)(x,y)=0,\,x,y,z\in V}.
\]

We also notice, that for $m\in\NN$,
\[
 K^{(m)}(\frak{h})=\qty(\Sym^mV^*\otimes K(\frak{h}))\cap\qty(\Sym^{m-1}V^*\otimes K^1(\frak{h})),
\]
where the intersection is to be understood as intersection of subspaces of $\Sym^{m-1}V^*\otimes V^*\otimes K(\frak{h})$.

Let us now define the tensor bundles
\begin{align*}
 K(M)&\coloneqq\bigsqcup_{x\in M}K(\End(T_xM))\subset\extp^2T^*M\otimes\End(TM),\\
 K^{(m)}(M)&\coloneqq\bigsqcup_{x\in M}K^{(m)}(\End(T_xM))\subset \Sym^mT^*M\otimes K(M).
\end{align*}

The Bianchi identities the curvature tensor $R$ of a torsion-free connection satisfies are thus equivalent to 
\begin{align*}
 R&\in\Gamma(K(M)),\\
 \nabla R&\in\Gamma(K^1(M)).
\end{align*}

\subsection{Holonomy Groups and Holonomy Algebras}\label{Subsec:HolonomyPrelim}

By a \emph{path} we mean a piecewise smooth curve $\gamma\colon[0,1]\longto M$. A \emph{loop} is a closed path.\\ $P_\gamma\colon T_{\gamma(0)}M\longto T_{\gamma(1)}M$ denotes the parallel translation map with respect to $\nabla$.

For $x\in M$ we define its \emph{holonomy group} based at $x$ as the subgroup
\[
 \Hol_x(\nabla)\coloneqq\qty{P_\gamma\;\vert\;\text{$\gamma$ is a loop based at $x$}}\subset\GL(T_xM)
\]

Since the manifold $M$ is assumed connected, notice that $\Hol_x(\nabla)$ is independent of the base point. Indeed, for any two points $x,y\in M$, 
\[
 \Hol_y(\nabla)=P_\gamma\Hol_x(\nabla)P_\gamma^{-1},
\]
for any path connecting $x$ with $y$.

Now, because the manifold $M$ is also assumed simply-connected, it holds \cite[Chapter~2, Proposition~ 2.2.4]{Joyce2007}, that $\Hol_x(\nabla)$ is in fact a connected Lie subgroup of $\GL(T_xM)$.

On the other hand, the parallel translation map concretely determines the horizontal paths in $F(M)$. Explicitly, it holds that a path $\overline{\gamma}\colon[0,1]\longto F(M)$ satisfies
\begin{equation}\label{Eq:HorizontalCurvesInFrameBundle}
 \overline{\gamma}'\in\ker\omega^\nabla\text{ if, and only if, }\overline{\gamma}(t)=P_{0,t}\circ\overline{\gamma}(0),
\end{equation}
where $P_{0,t}=P_{\gamma\vert_{[0,t]}}$, and $\gamma=\pi\circ\overline{\gamma}$.

For $u\in F(M)_x$ we define the \emph{holonomy group} of the connection form $\omega^\nabla$ as the subgroup of $G$ given by
\[
 \Hol_u(\omega^\nabla)\coloneqq\qty{u^{-1}\circ P_\gamma\circ u\;\vert\;\text{$\gamma$ is a loop based at $x$}}\subset G.
\]

Because of the fact that for every $x\in M$, $u\in F(M)_x$,
\[
 \Hol_x(\nabla)=u\Hol_u(\omega^\nabla)u^{-1},
\]
both notions of holonomy group are equivalent in this setting. Their respective Lie algebras are denoted by $\hol_x(\nabla)$, $\hol_u(\omega^\nabla)$.

The Reduction Theorem \cite[Chapter~2, Theorem~7.1]{kobayashi1963foundations} implies that for every $u\in F(M)_x$, $(P^\nabla(u),\pi\vert_{P^\nabla(u)},M)$ is a principal $\Hol_u(\omega^\nabla)$-subbundle of $F(M)$, where
\[
 P^\nabla(u)\coloneqq\qty{P_\gamma\circ u\;\vert\;\gamma\text{ is a path starting at }x}
\]
and the connection form reduces to $P^\nabla(u)$ in the sense that for every $q\in P^\nabla(u)$, 
\begin{equation}\label{Eq:ReductionOfOmegaNabla}
 \ker\omega^\nabla_q\subset T_qP^\nabla(u)
\end{equation}

The bundle $(P^\nabla(u),\pi\vert_{P^\nabla(u)},M;\Hol_u(\omega^\nabla))$ is called the \emph{holonomy bundle} of $\nabla$ through $u$.

One of the most important properties of the holonomy bundle of a connection is the fact that it is the smallest posible reduction of a principal bundle, in the sense that any reduction $(Q,\omega^Q)$ of $(F(M),\omega^\nabla)$ \cite[Chapter~4, Satz~4.1]{baum2009eichfeldtheorie} is in turn reducible to the holonomy bundle. Concretely it holds:
\begin{proposition}\label{Prop:SmallestReduction}
 Let $Q\subset F(M)$ be an immersed submanifold, let $(Q,\omega^Q)$ be a reduction of $(F(M),\omega^\nabla)$, and let $P^\nabla(u)$ denote the holonomy bundle of $\nabla$ through $u$. It holds:
 \begin{itemize}[noitemsep]
  \item[i)] For all $u\in Q$, $P^\nabla(u)\subset Q$,\\
  \item[ii)] $\omega^Q\vert_{TP^\nabla(u)}=\omega^\nabla\vert_{TP^\nabla(u)}$.
 \end{itemize}
\end{proposition}

Finally, the holonomy bundle allows to determine how big the holonomy group can be. 
\begin{proposition}[Ambrose-Singer holonomy Theorem]\cite[Chapter~4, Satz~4.5]{baum2009eichfeldtheorie}\label{Prop:Ambrose-Singer}
 Let $\omega^\nabla$ be the associated connection form to the connection $\nabla$.
 \begin{itemize}[noitemsep]
  \item[i)] For $u\in F(M)$ it holds:
   \[
    \hol_u(\omega^\nabla)=\gen\qty{F^\nabla_q(X,Y)\;\vert\;q\in P^\nabla(u),\,X,Y\in\ker\omega^\nabla_q}\subset\frg,
   \]
   where $F^\nabla\in\Omega^2(F(M),\frg)$ denotes the curvature form associated to $\omega^\nabla$.\\
   \item[ii)] Let $R^\nabla$ denote the curvature tensor associated to $\nabla$. For $x\in M$ it holds:
   \[
    \hol_x(\nabla)=\gen\qty{(P_\gamma^{-1}\cdot R^\nabla)(v,w)\;\vert\;v,w\in T_xM,\,\text{$\gamma$ is a path starting at $x$}},
   \]
   where $(P_\gamma^{-1}\cdot R^\nabla)(v,w)\coloneqq P_\gamma^{-1}\circ R^\nabla(P_\gamma v,P_\gamma w)\circ P_\gamma$.
 \end{itemize}
\end{proposition}

An immediate consequence of the Ambrose-Singer Theorem is the fact, that for the torsion-free connection $\nabla$, its curvature tensor satisfies for every $x\in M$:
\begin{align*}
 R^\nabla_x\in& K(\hol_x(\nabla)),\\
 (\nabla R^\nabla)_x\in&K^1(\hol_x(\nabla)).
\end{align*}

\subsection{The Exponential Framing}\label{Subsec:ExpFraming}
For $p\in M$ let $(\mathcal{U},(x^i))$ be a normal coordinate system centered at $p$, that is, there exists an open subset $\mathcal{V}\subset T_pM$ star-shaped around $0\in T_pM$ such that $\exp_p\colon\mathcal{V}\longto \mathcal{U}=\exp_p(\mathcal{V})$ is a diffeomorphism. For $v\in\mathcal{V}$ we denote by $\gamma_v\colon[0,1]\longto M$ the radial geodesic at $p$, i.e. $\gamma_v(t)=\exp_p(tv)$.

Normal coordinates allow us to define a special local frame:

\begin{proposition}\label{Prop:SmoothnessOfExpFraming}
 Let $\nabla$ be a connection on $TM$, and let $(\mathcal{U},(x^i))$ be a normal coordinate system centered at $p\in M$. Let $(\partial_1,\ldots,\partial_n)\colon\mathcal{U}\longto F(M)$ be the associated local frame. Set
 \begin{align*}
  \sigma\colon\mathcal{U}\longto& F(M)\\
  q\longsto&(\sigma_1(q),\ldots,\sigma_n(q))\coloneqq(P_{\gamma_v}e_1,\ldots,P_{\gamma_v}e_n),
 \end{align*}
 where $q=\gamma_v(1)$, $e_i\coloneqq\partial_i\vert_p$, and $P_{\gamma_v}\colon T_pM\longto T_qM$ denotes the parallel translation map. It holds that $\sigma\in\Gamma_{\mathcal{U}}(F(M))$. 
\end{proposition}

\begin{proof}
Let $(f_1,\ldots,f_n)\in\Gamma_{\mathcal{U}}(F(M))$ be any smooth local section with $f_i(p)=e_i$. Denote by $\widetilde{\Gamma}_{ij}^k$ the Christoffel symbols of $\nabla$ associated to this local frame.

 Because $(\mathcal{U},(x^i))$ is a normal coordinate system around $p$, we can write $\gamma_v(t)=tv$ for all $v\in\mathcal{V}$, $t\in[0,1]$. Set $\sigma_i(v)=a_i^k(v)f_k(v)$.
 
 Since the curve $t\longsto P_{\gamma_{tv}}e_i=\sigma_i(tv)$ is $\gamma_v$-parallel, we obtain, the coefficients $a_i^k$ are a solution for the initial value problem
 \[
  \begin{cases}
   \dv{}{t}a_i^k(tv)+\widetilde{\Gamma}_{j\ell}^k(tv)v^ja_i^\ell(tv)=0,\\
   a_i^k(0)=\delta_i^k
  \end{cases}
 \]
 
 With this we observe that the functions $\widetilde{a}_i^k(t,v)\coloneqq a_i^k(tv)$ are a solution to the initial value problem
 \[
  \begin{cases}
   \pdv{}{t}\widetilde{a}_i^k(t,v)+\widetilde{\Gamma}_{j\ell}^k(tv)v^j\widetilde{a}_i^\ell(t,v)=0,\\
   \widetilde{a}_i^k(0,v)=\delta_i^k
  \end{cases}
 \] 
 which admits a unique smooth solution $\hat{a}_i^k$, for every $1\leq i,k\leq n$. 
 
One readily observes that for every $c>0$, the maps $(t,v)\longsto\hat{a}_i^k(t/c,cv)$ are also smooth solutions for the above initial value problem, which thus implies $\hat{a}_i^k(t/c,cv)=\hat{a}_i^k(t,v)$ and so we obtain $a_i^k(v)=\hat{a}_i^k(1,v)$, which yields the smoothness of each of the maps $\sigma_i\colon\mathcal{U}\longto TM$.
\end{proof}

\begin{definition}\label{Prop:ExponentialFraming}
 An \emph{exponential framing} at $p$ is the local frame $\sigma\in\Gamma_{\mathcal{U}}(F(M))$ introduced in Proposition \ref{Prop:SmoothnessOfExpFraming}, where $\mathcal{U}$ is the domain of a normal coordinate system centered at $p$.
\end{definition} 

In terms of the connection form $\omega^\nabla\in\Omega^1(F(M),\frg)$ induced by the connection $\nabla$, the exponential framing $\sigma$ takes the form
\begin{equation}\label{Eq:ExpFramingShort}
 \sigma(\gamma_v(t))=P_{\gamma_{tv}}\circ u_0\eqqcolon\overline{\gamma}_{u_0}(t),
\end{equation}
where $u_0\in F(M)_p$ denotes the frame which maps the $i$-th standard basis vector of $V$ to the basis vector $e_i$ of $T_pM$. In other words, for all $v\in\mathcal{V}$, $t\in[0,1]$, $\sigma(\gamma_v(t))$ lies in $P^\nabla(u_0)$, the holonomy bundle of $\nabla$ through $u_0$, and because of the fact that $\exp_p\colon\mathcal{V}\longto\mathcal{U}$ is a diffeomorphism, we obtain in sum:
\begin{equation}\label{Eq:ExpFramingInHolBundle}
 \sigma(\mathcal{U})\subset P^\nabla(u_0).
\end{equation}

We finish up by making a final observation derived from the fact that the exponential framing takes values in the holonomy bundle of the connection. Let $\evf$ denote the Euler vector field of the manifold. That is, the vector field associated to the flow $\phi_t\colon V\longto V$, with $\phi_t(v)=e^tv$. Notice that for any $q=\gamma_v(1)\in\mathcal{U}$, 
\begin{equation}\label{Eq:OmegaHatXiVanishes}
\hat{\omega}(\evf_q)\coloneqq\sigma^*\omega^\nabla(\evf_q)=0.
\end{equation}

Indeed, from \eqref{Eq:ExpFramingShort} we obtain
 \[
  \hat{\omega}(\evf_{q})=\omega^\nabla(\dd_q\sigma\evf_q)=
  \omega^\nabla(\dd\sigma(\gamma_v'(1)))=
  \omega^\nabla(\overline{\gamma}_{u_0}'(1))\overset{\scriptscriptstyle{\eqref{Eq:HorizontalCurvesInFrameBundle}}}{=}0.
 \]
 
 \section{Necessary Conditions for the Existence of Torsion-free Connections}\label{Sec:NecConditions}
 
 In this section we analyze in some depth a variety of the properties enjoyed by the curvature tensor of a torsion-free connection.
 
 \subsection{The Principal Bundle Picture}\label{Subsec:PrincBundlePicture}
 For $m\in\NN_0$, let $\rho_m\colon G\longto\GL(\bigotimes^mV^*\otimes\frak{g})$ denote the representation given by
 \begin{equation}\label{Eq:RepRhom}
  \rho_m(T)(\alpha_1\otimes\cdots\otimes\alpha_m\otimes X)\coloneqq\alpha_1\circ T^{-1}\otimes\cdots\otimes\alpha_m\circ T^{-1}\otimes\Ad(T)X.
 \end{equation}
 
 It is easy to see that for $m\in\NN_0$, $\rho_{m+2}$ restricts to a representation of the space $K^{(m)}(\frak{g})$.
 
 This representation allows us to establish the vector bundle isomorphism
 \[
  K^{(m)}(M)\simeq F(M)\times_{\rho_{m+2}} K^{(m)}(\frak{g})\hspace{1cm}m\in\NN_0,
 \]
 where the right-hand side denotes the associated vector bundle of $F(M)$ with typical fiber $K^{(m)}(\frak{g})$. That is, the vector bundle with total space $F(M)\times K^{(m)}(\frak{g})/\sim$, and the equivalence relation $\sim$ is given by $(v,\phi)\sim(v\circ T,\rho_m(T^{-1})\phi)$, for all $T\in G$, $v\in F(M)$, $\phi\in K^{(m)}(\frak{g})$. This isomorphism thus enables us to establish the one-to-one correspondence 
 \begin{equation}\label{Eq:SectionsOfKmCorrespondence}
  \Gamma(K^{(m)}(M))\equalhat C^\infty(F(M),K^{(m)}(\frak{g}))^{G},
 \end{equation}
 the set of smooth $G$-equivariant maps of type $\rho_{m+2}$ from $F(M)$ to $K^{(m)}(\frak{g})$.
   
 Similarly, the standard and adjoint representations of $G=\GL(V)$ induce the vector bundle isomorphisms
 \begin{align*}
  TM&\simeq F(M)\times_{\underline{\rho}}V,\\
  \End(TM)&\simeq F(M)\times_{\mathrm{Ad}}\frak{g}\eqqcolon\Ad(F(M)).
 \end{align*}

In terms of the curvature tensor, notice that the curvature form $F^\nabla$ satisfies the identity:
\begin{equation}\label{Eq:F-RThetaTheta}
 F^\nabla=R(\theta^\nabla,\theta^\nabla),
\end{equation}
where $R\in C^\infty(F(M),K(\frak{g}))^G$ denotes the unique $G$-equivariant map associated to the curvature tensor $R^\nabla$ given by the correspondence \eqref{Eq:SectionsOfKmCorrespondence}.

Given the fact the results in this work are local, we can restrict ourselves, without loss of generality, to the case in which $M=U$, where $U$ denotes an open subset of $V$, which is also star-shaped around 0 such that the exponential map $\exp_0\colon U\subset V\cong T_0U\longto U$ associated to a (non-trivial) torsion-free connection $\nabla$ on $TU=U\times V$ is the identity. By choosing a normal coordinate system around a fixed $p$ on the manifold $M$ we can always reduce to this case. In this setting notice the vector bundle isomorphisms
 \begin{align*}
  F(U)&\simeq U\times G,\\
  K^{(m)}(U)&\simeq U\times K^{(m)}(\frak{g}).
 \end{align*} 
 
 The vector bundle isomorphism $F(U)\times_{\rho_{m+2}} K^{(m)}(\frak{g})\simeq K^{(m)}(U)$ is then explicitly given by
 \begin{align*}
  F(U)\times_{\rho_{m+2}} K^{(m)}(\frak{g})&\longto K^{(m)}(U),\\
  \qty[(v,g),\phi]&\longsto(v,\rho_{m+2}(g)\phi)
 \end{align*}
 
 By means of this isomorphism we plainly write $K^{(m)}(U)=F(U)\times_{\rho_{m+2}} K^{(m)}(\frak{g})$.
 
 The standard trivialization of the frame bundle is given by the local section
 \begin{align*}
  \ve\colon U&\longto U\times G,\\
  v&\longsto(v,\Id_V)
 \end{align*}
 
 We also notice that for any section $\phi\in\Gamma(K^{(m)}(U))$ there exists a unique smooth map \\$\Phi\colon U\longto K^{(m)}(\frak{g})$ such that
 \[
  \phi(v)=[\ve(v),\Phi(v)].
 \]
 
 Because of Proposition \ref{Prop:SmoothnessOfExpFraming}, we obtain the smooth map $h\colon U\longto G$ given by $h(v)=P_{\gamma_v}$. In this case the exponential framing $\sigma$ takes the form
 \[
  \sigma(v)=(v,h(v)).
 \]
 
 From the fact that $\sigma(v)=\ve(v)\cdot h(v)$ we thus obtain 
 \[
  \phi(v)=[\sigma(v),\rho_{m+2}(h(v)^{-1})\Phi(v)]\eqqcolon[\sigma(v),\qty(\rho_{m+2}(h^{-1})\Phi)(v)].
 \]
 
Furthermore, $\sigma=\varepsilon\cdot h$ implies the standard change of gauge relations
\begin{eqnarray}
 \hat{\omega}&=&\Ad(h^{-1})\circ\varepsilon^*\omega^\nabla+h^*\mu_G\eqqcolon\Ad(h^{-1})\circ\Gamma+h^*\mu_G,\\
 \hat{F}&\coloneqq&\sigma^*F=\Ad(h^{-1})\circ\varepsilon^* F\eqqcolon\Ad(h^{-1})\circ \underline{F},\\
 \hat{\theta}&\coloneqq&\sigma^*\theta=h^{-1}\varepsilon^*\theta\eqqcolon h^{-1}\underline{\theta},\label{Eq:hGauge}
\end{eqnarray}
where $\mu_G$ denotes the left-invariant Maurer-Cartan form of $G$.

We observe that for $v\in U$, $w\in V=T_vU$,
\[
 \underline{\theta}_v(w)=w.
\]

In other words, it holds
\[
 \underline{\theta}=\dd v,
\]
where $v$ denotes the generic point in $U$. That is, $v\colon U\longto U$ is the identity map.

Now, by pulling back $\dd\theta+\omega\wedge\theta=0$ we obtain $\Gamma\wedge\dd v=0$. Which amounts to the fact that $\Gamma$ actually takes values in $\Sym^2V^*\otimes V$.

Moreover, by pulling back the first Bianchi identity we also obtain
\[
 \underline{F}\wedge\dd v=0,
\]
which thus implies that $\underline{F}$ takes values in $K(\frak{g})$.

Thus, in this context we obtain that the relation \eqref{Eq:F-RThetaTheta} in these trivializations reads:
\begin{align}
 \underline{F}=&R^\nabla,\\
 \hat{F}=&\rho_2(h^{-1})R(\hat{\theta},\hat{\theta})\eqqcolon\hat{R}(\hat{\theta},\hat{\theta})\label{Eq:Pulled-back F-RThetaTheta}
\end{align}

Notice as well that along the standard trivialization $\ve\colon U\longto U\times G$ given by $\ve(q)=(q,\Id_V)$, we readily obtain the relation
 \begin{equation}\label{Eq:CovariantDerR}
  \nabla R^\nabla=\dd R^\nabla+(\rho_2)_*(\Gamma)\wedge R^\nabla.
 \end{equation}
 
 A relevant observation is the fact that under general gauges, a similar structural equation is satisfied:
 \begin{proposition}\label{Prop:CurvatureAndGauges}
  Let $U\subset V$ be an open set containing $0$, let $g\colon U\longto G$ be a smooth map, and denote by $s$ the associated local section $s\colon U\longto F(U)=U\times G$ given by $s(q)\coloneqq(q,g(q))$. Then $\overline{R}\coloneqq\rho_2(g^{-1})R^\nabla$, $\overline{\omega}\coloneqq s^*\omega^\nabla$ satisfy the structure equation
  \begin{equation}\label{Eq:CurvAndGauges}
   (\dd\overline{R}+(\rho_2)_*(\overline{\omega})\wedge\overline{R})(g)=\rho_3(g^{-1})\nabla R^\nabla,
  \end{equation}
  
  where $R^\nabla$ denotes the curvature tensor of the torsion-free connection $\nabla$, and $\omega^\nabla\in\Omega^1(F(U),\frak{g})$ denotes its associated connection form. In particular, the map $(\dd\overline{R}+(\rho_2)_*(\overline{\omega})\wedge\overline{R})(g)$ takes values in $K^1(\frak{g})$.
 \end{proposition}
 
 The proof of Proposition \ref{Prop:CurvatureAndGauges} relies on the following elementary 
 \begin{lemma}\label{Lem:AuxiliaryGaugeLemma}
  Let $H$ be a Lie group, $V$ a vector space, and $\rho\colon H\longto\GL(V)$ a representation. Let $U$ be a manifold, $X$ a vector field on $U$, and $h\colon U\longto H$, $f\colon U\longto V$ smooth functions. It holds that
  \[
   \dd(\rho(h^{-1})f)(X)=\rho(h^{-1})\qty(\dd f(X)-\rho_*(\Ad h(h^*\mu_H(X)))f),
  \]
  where $\mu_H\in\Omega^1(H,\frak{h})$ denotes the left-invariant Maurer-Cartan form of the group $H$.
 \end{lemma}
 \begin{proof}[Proof of Proposition \ref{Prop:CurvatureAndGauges}]
  By making use of Lemma \ref{Lem:AuxiliaryGaugeLemma} we obtain for every $x\in V$:
  \[
   \dd\overline{R}(x)=\rho_2(g^{-1})\dd \Rnabla(x)-\rho_2(g^{-1})(\rho_2)_*(\Ad g(g^*\mu_G(x)))\Rnabla
  \]
  
  Because of the identity $\nabla\Rnabla=\dd\Rnabla+(\rho_2)_*(\Gamma)\wedge \Rnabla$, we then get
 \begin{align*}
  \dd\overline{R}(x)&=\rho_2(g^{-1})\nabla_x\Rnabla-\rho_2(g^{-1})(\rho_2)_*(\Gamma(x))\Rnabla-\rho_2(g^{-1})(\rho_2)_*(\Ad g(g^*\mu_G(x)))\Rnabla\\
  =&\rho_2(g^{-1})\nabla_x\Rnabla-\rho_2(g^{-1})(\rho_2)_*(\Ad g(\Ad g^{-1}\Gamma+g^*\mu_G)(x))\Rnabla\\
  =&\rho_2(g^{-1})\nabla_x\Rnabla-\rho_2(g^{-1})(\rho_2)_*(\Ad g\circ\overline{\omega}(x))\Rnabla
 \end{align*}
 
 A direct computation also shows:
 \[
  \rho_2(g^{-1})(\rho_2)_*(\Ad g\circ\overline{\omega}(x))\Rnabla=(\rho_2)_*(\overline{\omega}(x))\overline{R}=((\rho_2)_*(\overline{\omega})\wedge\overline{R})(x).
 \]
 
 It thus holds, by the mere definition of the representation $\rho_m$ in \eqref{Eq:RepRhom},
 \[
  (\dd\overline{R}+(\rho_2)_*(\overline{\omega})\wedge\overline{R})(g\cdot x)=\rho_2(g^{-1})\nabla_{g\cdot x}\Rnabla=(\rho_3(g^{-1})\nabla \Rnabla)(x)
 \]
 \end{proof}
 
 \subsection{An Initial Value Problem}\label{Subsec:AnIVP}
 
 The goal of this part is to expose how the results of the previous sections come together and provide us with conditions a curvature map must satisfy in order to have a chance of arising from the curvature tensor of a torsion-free connection.

Firstly let $W$ denote a real vector space (for the purposes of this work, it is enough to consider the case $W\subset\bigotimes^mV^*\otimes\frak{g}$, for some $m\in\NN_0$), and $\mathcal{L}_\evf\colon\Omega^\bullet(U,W)\longto\Omega^\bullet(U,W)$ denote the Lie derivative along the Euler vector field $\evf$. In terms of the flow $\phi\colon\RR\times V\longto V$ of $\evf$, it holds that
\[
  \dv{}{t}\phi_t^*\eta=\phi_t^*\mathcal{L}_\evf\eta.
 \]
 
 On the other hand, since $\phi_0^*=\Id_{\Omega^\bullet(U,W)}$ and $\lim_{t\to-\infty}\phi_t=0$, we obtain
 \begin{equation}\label{Eq:MotivationForI}
  \eta=\int_{-\infty}^0\dv{}{t}\phi_t^*\eta\,\dd t.
 \end{equation}
 
 For $k\in\NN_0$, let $I\colon\Omega^k(U,W)\longto\Omega^k(U,W)$ be the \emph{integration map} defined by
  \[
   I\eta\coloneqq
   \begin{cases}
   \int_{-\infty}^0\phi_t^*\eta\,\dd t&k\geq1,\\
   \int_{-\infty}^0\phi_t^*(\eta-\eta_0)\,\dd t&k=0,
   \end{cases}
 \]
 where $\eta_0\in C^\infty(U,W)$ denotes the constant map $\eta_0\equiv\eta(0)$.
 
 After a change of variable, we obtain that for $x\in U$,
 \begin{equation}\label{Eq:ExplicitI}
  (I\eta)_x=
  \begin{cases}
  \int_0^1t^{k-1}\eta_{tx}\,\dd t&k\geq1,\\
  \int_0^1\frac{\eta_{tx}-\eta_0}{t}\,\dd t&k=0.
  \end{cases}
 \end{equation}
 
 The integration map $I$ is easily seen to commute with both the exterior and interior derivatives\\ $\dd\colon\Omega^k(U,W)\longto\Omega^{k+1}(U,W)$, $\iota_\evf\colon\Omega^k(U,W)\longto\Omega^{k-1}(U,W)$.
 
 By using the fact the $\mathcal{L}_\evf=\iota_\evf\dd+\dd\iota_\evf$ we obtain for $\eta\in\Omega^k(U,W)=C^\infty(U,\extp^kV^*\otimes W)$:
 \begin{equation}\label{Eq:LieDerivativeAlongE}
  \mathcal{L}_\evf\eta=k\eta+\eta'(\evf),
 \end{equation}
 where $\eta'$ denotes the differential of $\eta$ considered as a $\extp^kV^*\otimes W$-valued 0-Form.
 
 For $\eta\in\Omega^k(U,\frak{g})$ we define the form $\eta\cdot\evf\in\Omega^k(U,V)$ as
 \begin{equation}\label{Eq:EtaTimesE}
  (\eta\cdot\evf)(v_1,\ldots,v_k)\coloneqq\eta(v_1,\ldots,v_k)\evf
 \end{equation}

 By virtue of \eqref{Eq:MotivationForI}, we obtain for $\eta\in\Omega^k(U,W)$, $k\in\NN_0$ the relations 
 \begin{equation}\label{Eq:LeI}
  I\mathcal{L}_\evf\eta=\mathcal{L}_\evf I\eta=
  \begin{cases}
   \eta&k\geq1,\\
   \eta-\eta_0&k=0.
  \end{cases}
 \end{equation}
 
 Writing the formal power series of $\eta\in\Omega^k(U,W)$
 \[
  \eta\approx\sum_{m\geq0}\eta^{(m)},
 \]
 where $\eta^{(m)}\in\Sym^mV^*\otimes\extp^kV^*\otimes W$, we readily obtain the formal power series expansion
 \begin{equation}\label{Eq:FormalPowerSeriesI}
  I\eta\approx\begin{cases}
   \sum_{m\geq0}\tfrac{1}{m+k}\eta^{(m)}&k\geq1,\\
   \sum_{m\geq1}\tfrac{1}{m}\eta^{(m)}&k=0.
  \end{cases}
 \end{equation}
 
 The Lie derivative along the Euler vector field $\evf$ and its associated map $I$ can be used to characterize the pull-backed forms $\hat{\theta},\hat{\omega}$:
 \begin{proposition}\label{Prop:ThetaHat-OmegaHat}
  With the above notation, it holds:
  \begin{itemize}[noitemsep]
   \item[i)] $\hat{\theta}=\dd v+I(\hat{\omega}\cdot\evf)$. In other words, for the map $h^{-1}\colon U\longto G$ it holds: $h^{-1}=\Id+I(\hat{\omega}\cdot\evf)$, where $h^{-1}$ denotes the smooth map induced by parallel translation along radial geodesics given in \eqref{Eq:hGauge} \\
   \item[ii)] $\hat{\omega}=I(\iota_\evf\hat{F})=I(\hat{R}(\evf,\hat{\theta}))$.
  \end{itemize}
 \end{proposition}
 
 \begin{proof}
  To \emph{i)}: Notice that, from the fact $h^{-1}(q)q=q$ for all $q\in U$, we obtain at once $\hat{\theta}(\evf)=\evf$. Thus, together with the fact that $\dd\hat{\theta}+\hat{\omega}\wedge\hat{\theta}=0$ it follows that:
  \[
   \mathcal{L}_\evf\hat{\theta}=\dd\iota_\evf\hat{\theta}+\iota_\evf\dd\hat{\theta}=\dd v-\iota_\evf(\hat{\omega}\wedge\hat{\theta})=\dd v-\iota_\evf(\hat{\omega}\wedge\dd v)=\dd v+\hat{\omega}\cdot\evf,
  \]
  where the last equality follows from the fact that $\iota_\evf\hat{\omega}=0$.
  
  Since $\hat{\theta}\in\Omega^1(U,V)$, the map $I$ is in fact the inverse of $\mathcal{L}_\evf$. The claim follows now from the linearity of $I$ and the fact that $I(\dd v)=\dd v$.
  
  To \emph{ii)}: From the fact that $\hat{F}=\dd\hat{\omega}+\hat{\omega}\wedge\hat{\omega}=\hat{R}(\hat{\theta},\hat{\theta})$ and because $\iota_\evf\hat{\omega}=0$, it thus holds:
  \[
   \iota_\evf\hat{F}=\hat{R}(\hat{\theta}(\evf),\hat{\theta})=\hat{R}(\evf,\hat{\theta})=\iota_\evf\dd\hat{\omega}=\mathcal{L}_\evf\hat{\omega},
  \]
  from which the claim immediately follows. 
 \end{proof}
 
 We finish this section by stating a simple corollary of Proposition \ref{Prop:ThetaHat-OmegaHat}, which turns out to be a key ingredient for the formulation of the main result.
 \begin{proposition}\label{Prop:IVPThetaHat}
 The pulled-back form $\hat{\theta}$ is a solution to the singular initial value problem
 \[
  \begin{cases}
   \mathcal{L}_\evf(\mathcal{L}_\evf-\Id)\theta=\hat{R}(\evf,\theta)\evf,\\
   \theta_0=\Id,\,\dd\theta_0=0.
  \end{cases}
 \]
 \end{proposition}
 \begin{proof}
  From \emph{i)} in Proposition \ref{Prop:ThetaHat-OmegaHat} we obtain:
  \begin{align*}
   \hat{\theta}_0=&\Id+I(\hat{\omega}\cdot\evf)_0=\Id+(\hat{\omega}\cdot\evf)_0=\Id,\\
   \dd\hat{\theta}_0=&\dd I(\hat{\omega}\cdot\evf)_0=I(\dd(\hat{\omega}\cdot\evf))_0=\int_0^1t\dd(\hat{\omega}\cdot\evf)_0\,\dd t=\frac{1}{2}\dd(\hat{\omega}\cdot\evf)_0=0,
  \end{align*}
  where the last equation follows from the fact that
  \[
   \dd(\hat{\omega}\cdot\evf)=(\dd\hat{\omega})\cdot\evf-\hat{\omega}\wedge\dd v,
  \]
  and, because of \emph{ii)} in Proposition \ref{Prop:ThetaHat-OmegaHat}, $\hat{\omega}_0=0$.
  
  It further holds:
  \begin{align*}
   \mathcal{L}_\evf^2\hat{\theta}=&\mathcal{L}_\evf(\dd v+\hat{\omega}\cdot\evf)\\
   =&\dd v+\hat{\omega}\cdot\evf+(\mathcal{L}_\evf\hat{\omega})\cdot\evf\\
   =&\mathcal{L}_\evf\hat{\theta}+\hat{R}(\evf,\hat{\theta})\evf,
  \end{align*}
  from which the claim immediately follows.
 \end{proof}
 
 \section{Main Result}\label{Sec:MainResult}
 
 In this section we prove our main result, Theorem \ref{Thm:MainThm}, which basically states that solutions to the initial value problem introduced in Proposition \ref{Prop:IVPThetaHat} turn out to be essential for the existence of torsion-free connections. Throughout this section $U$ is going to denote an open, star-shaped around $0\in V$ subset.
 
 Before stating the proof of Theorem \ref{Thm:MainThm}, we need a couple of preliminary results.
 \begin{proposition}\label{Prop:MainIVP}
  Let $S\colon U\longto K(\frak{g})$ a real analytic map. Then, the singular initial value problem
  \begin{equation}\label{Eq:MainIVP}
   \begin{cases}
    \mathcal{L}_\evf(\mathcal{L}_\evf-\Id)\theta=S(\evf,\theta)\evf,\\
    \theta_0=\Id,\,\dd\theta_0=0
   \end{cases}
  \end{equation}
  admits a unique real analytic solution.
 \end{proposition}
 
 This result can be proved by means of the Theorem of Cauchy-Kovalevskaya in the singular setting. However, because of the nature of this particular initial value problem, one can prove the existence and uniqueness of the analytic solution to \eqref{Eq:MainIVP} by purely elementary considerations. We follow this path, but in order to do so, some auxiliary results are needed.
 
 Firstly we recall the following criterion for real-analyticity of a function.
 \begin{proposition}\label{Prop:DerOfAnaFunctions}\cite[Chapter~2, Proposition~2.2.10]{krantz1992primer}
  Let $U\subset\RR^m$ be an open neighborhood of $0$. A smooth map $F\colon U\longto\RR^\ell$ is real-analytic if, and only if, for every $x_0\in U$ there exists an open neighborhood $U_0$ of $x_0$, and positive constants $C,r$ such that for every multi-index $\mu\in\NN_0^m$, $x\in U_0$:
  \[
   \norm{\frac{1}{\mu!}D^\mu F(x)}\leq\frac{C}{r^{\abs{\mu}}},
  \]
  where $\norm{\cdot}$ denotes the Euclidean norm of $\RR^\ell$, and for $\mu=(\mu_1,\ldots,\mu_m)$, $\mu!=\mu_1!\cdots\mu_m!$, $\abs{\mu}=\sum_{i=1}^m\mu_i$.
 \end{proposition}
 
 Before we state the following auxiliary result we notice that the symmetric powers $\Sym^kV$ can be endowed with an inner product coming from the one in $V$. Explicitly we define \\ $\langle\cdot,\cdot\rangle\colon\Sym^k V\times\Sym^k V\longto\RR$ as the bilinear map that on monomials is given by
 \[
  \langle v_1\cdots v_k,w_1\cdots w_k\rangle\coloneqq\frac{1}{k!}\perm(\langle v_i,w_j\rangle)_{1\leq i,j\leq k}\coloneqq\frac{1}{k!}\sum_{\sigma\in S_k}\prod_{i=1}^k\langle v_i,w_{\sigma(i)}\rangle.
 \]
 
 The induced norm by this normalized inner product satisfies that, for any $v\in V$, it holds $\norm{v^k}=\norm{v}^k$.
 
 Proposition \ref{Prop:DerOfAnaFunctions} allows to establish the following
 \begin{proposition}\label{Prop:SeveralVariablesToOneVariable}
 Let $V$, $(W,\norm{\cdot}_W)$ be finite-dimensional vector spaces, say $\dim V=n$. If the function $F\colon V\longto W$ is analytic near the origin, then 
 \[
 f(t)\coloneqq\sum_{m\geq0}\norm{F^{(m)}}t^m
 \]
 defines an analytic function near $0\in\RR$, where $F^{(m)}\in\Sym^mV^*\otimes W$ denotes the unique polynomial that satisfies
 \[
  F^{(m)}(v^m)=\sum_{\abs{\mu}=m}\frac{1}{\mu!}D^\mu F(0)v^\mu,
 \]
 and $\norm{\cdot}$ denotes the norm on $\Sym^mV^*\otimes W$ given by
 \[
  \norm{Q}\coloneqq\max_{\norm{v}=1}\norm{Q(v^m)}_W.
 \]
 \end{proposition}
 \begin{proof}
  Because of Proposition \ref{Prop:DerOfAnaFunctions}, there exists positive constants $C,r$ such that for every $x$ in a sufficiently small open ball around $0$, and every multi-index $\mu\in\NN_0^n$,
  \[
   \norm{\frac{1}{\mu!}D^\mu F(x)}_W\leq\frac{C}{r^{\abs{\mu}}}.
  \]
  
  For every unitary $v\in V$ it thus hold:
  \begin{align*}
   \norm{F^{(m)}(v^m)}_W\leq&\sum_{\abs{\mu}=m}\norm{\frac{1}{\mu!}D^\mu F(0)}_W\norm{v^\mu}\overset{\scriptscriptstyle{\norm{v^\mu}=\norm{v}^{\abs{\mu}}}}{=}\sum_{\abs{\mu}=m}\norm{\frac{1}{\mu!}D^\mu F(0)}_W\\
   \leq&\sum_{\abs{\mu}=m}\frac{C}{r^m}\\
   =&\frac{C}{r^m}\dim(\Sym^m V)=\frac{C}{r^m}\binom{n+m-1}{m}\\
\leq&\frac{C}{r^m}(m+1)^{n-1}
  \end{align*}
  and so we conclude, that for every $m\in\NN_0$,
  \[
   \norm{F^{(m)}}\leq\frac{C}{r^m}(m+1)^{n-1}.
  \]
  
  Now, because for $\abs{t}<r$, $k\in\NN_0$, it holds that
  \[
   \sum_{m\geq0}(m+1)^k\qty(\frac{t}{r})^m=\qty(t\dv{}{t}+\Id)^k\frac{r}{r-t},
  \]
  we thus obtain
  \[
   \sum_{m\geq0}\abs{\norm{F^{(m)}}t^m}\leq C\sum_{m\geq0}(m+1)^{n-1}\qty(\frac{\abs{t}}{r})^m<\infty,
  \]
  which then implies, that for every $\ve>0$, $\sum_{m\geq0}\norm{F^{(m)}}t^m$ uniformly converges on $[-r+\ve,r-\ve]$, thus yielding $f\in C^\omega((-r,r))$.
 \end{proof}
 
 We now have all the necessary tools to prove Proposition \ref{Prop:MainIVP}
 \begin{proof}[Proof of Proposition \ref{Prop:MainIVP}]
  Let $\theta=\sum_{m\geq0}\theta^{(m)}$ be a formal solution to \eqref{Eq:MainIVP}. From the initial data we obtain:
  \begin{multicols}{2}
   \noindent
   \[
    \theta^{(0)}=\Id,
   \]
   \[
    \theta^{(1)}=0.
   \]
  \end{multicols}
  
  In light of \eqref{Eq:FormalPowerSeriesI} it holds:
  \[
   \mathcal{L}_\evf(\mathcal{L}_\evf-\Id)\theta=\mathcal{L}_\evf\sum_{m\geq1}m\theta^{(m)}=\sum_{m\geq1}m(m+1)\theta^{(m)}=\sum_{m\geq0}(m+1)(m+2)\theta^{(m+1)}.
  \]
  
  Because for every $q\in U$, $S_q=\sum_{m\geq0}S^{(m)}(q^m)$, we conclude that the formal power series of the map $S(\evf,\theta)\evf$ satisfies
  \[
   (S(\evf,\theta)\evf)_q=\sum_{m\geq0}\sum_{k+\ell=m-2}S^{(k)}(q^k)(q,\theta^{(\ell)}(q^\ell))q\in V^*\otimes V.
  \]
  
  With this we thus obtain, that for every $q\in U$, $m\in\NN_0$,
  \begin{eqnarray}
    \theta^{(0)}&=&\Id,\\
    (m+1)\theta^{(m+1)}(q^{m+1})&=&\frac{1}{m+2}\sum_{k+\ell=m-1}S^{(k)}(q^k)(q,\theta^{(\ell)}(q^\ell))q\label{Eq:ThetaRecurrence}
  \end{eqnarray}
  
  For any $q\in U$ with $\norm{q}\leq1$, and by endowing $\extp^2V\otimes V$ with the natural norm induced by the inner product in $V$, we obtain for every $m\in\NN_0$
  \begin{align*}
   (m+1)\norm{\theta^{(m+1)}(q^{m+1})}\leq&\frac{1}{m+2}\sum_{k+\ell=m-1}\norm{S^{(k)}(q^k)}\norm{\theta^{(\ell)}(q^\ell)}\norm{q}^2\\
   \leq&\sum_{k+\ell=m-1}\norm{S^{(k)}}\norm{\theta^{(\ell)}}\\
   \leq&\sum_{k+\ell=m}\norm{S^{(k)}}\norm{\theta^{(\ell)}},
  \end{align*}
  implying thus for every $m\in\NN_0$
  \begin{equation}\label{Eq:ThetaNormEstimate}
   (m+1)\norm{\theta^{(m+1)}}\leq\sum_{k+\ell=m}\norm{S^{(k)}}\norm{\theta^{(\ell)}}.
  \end{equation}
  
  Now, since $S\in C^\omega(U,K(\frak{g}))$, the map $s(t)\coloneqq\sum_{m\geq0}\norm{S^{(m)}}t^m$ is analytic near $0$, according to Proposition \ref{Prop:SeveralVariablesToOneVariable}. Notice also that $\vartheta(t)\coloneqq\sum_{m\geq0}\norm{\theta^{(m)}}t^m$ is absolutely convergent near $0$, since from \eqref{Eq:ThetaNormEstimate} we obtain that for $t>0$,
  \[
   \vartheta'(t)\leq s(t)\vartheta(t),
  \]
  and so, by Grönwall's lemma we obtain
  \[
   \vartheta(t)\leq\vartheta(0)e^{\int_0^t s(u)\,\dd u}=e^{\int_0^ts(u)\,\dd u},
  \]
  which due to the analyticity of $t\longsto e^{\int_0^ts(u)\,\dd u}$ near 0, implies the absolute convergence of $\vartheta$ close to $0$.
  
  Thus, for $x\in B_\rho(0)$, for sufficiently small $\rho$, we have proved the absolute convergence of
  \[
   \sum_{m\geq0}\theta^{(m)}(x^m),
  \]
  as well as its uniform convergence on compact subsets of $B_\rho(0)$. This implies together with the recursion \eqref{Eq:ThetaRecurrence}, the analytic solution $\theta$ is uniquely given, yielding thus the claim.
 \end{proof}
 
 The analytic solution to the above initial value problem admits a concrete characterization in terms of a special $\frak{g}$-valued 1-form:
 \begin{proposition}\label{Prop:ThetaProperties}
  Let $S\colon U\longto K(\frak{g})$ be a real analytic map, let $\theta$ be the real analytic solution to the singular initial value problem \eqref{Eq:MainIVP}. Set $\omega=I(S(\evf,\theta))$. It then holds:
  \begin{itemize}[noitemsep]
   \item[i)] $\theta=\dd v+I(\omega\cdot\evf)$,\\
   \item[ii)] $\mathcal{L}_\evf(\dd\theta+\omega\wedge\theta)=(\dd\omega+\omega\wedge\omega-S(\theta,\theta))\cdot\evf$.
  \end{itemize}
 \end{proposition}
 
 Before stating the proof of the Proposition, we make the following observation:
 \begin{lemma}\label{Lem:FixedPointsOfLieE}
  Let $U\subset V$ a star-shaped around $0$ open subset, $W$ a vector space. Then\\ $\eta\in\Omega^1(U,W)$ satisfies $\mathcal{L}_\evf\eta=\eta$ if, and only if, $\eta\equiv\eta(0)$.
 \end{lemma}
 \begin{proof}
  That the constant form $\eta(0)$ satisfies the claim is easily seen from the definition. Suppose now $\mathcal{L}_\evf\eta=\eta$. Because of \eqref{Eq:LieDerivativeAlongE}, this implies, $\eta'(\evf)=0$. This is equivalent to the fact that $e^i\partial_i\eta\equiv0$. From this we thus conclude, that for all $j$, $\partial_j\eta=-e^i\partial_{ij}\eta$, and so we obtain in particular, $\partial_j\eta(0)=0$. 
  
  Now, for $q\in U$ and $t>0$ we notice that the smooth map $t\longto\eta(tq)$ is constant, since
  \[
   \dv{}{t}\eta(tq)=q^j\partial_j\eta(tq)=\frac{1}{t}(e^j\partial_j\eta)_{tq}=0.
  \]
  
  In fact, this map is everywhere constant, since,
  \[
   \eval{\dv{}{t}}_0\eta(tq)=q^j\partial_j\eta(0)=0.
  \]
  
  This readily yields the claim.
 \end{proof}
 
 \begin{proof}[Proof of Proposition \ref{Prop:ThetaProperties}]
 
 To \emph{i)}: Set $\tilde{\theta}\coloneqq\dd v+I(\omega\cdot\evf)$. It thus holds:
 \begin{align*}
  \mathcal{L}_\evf^2\tilde{\theta}=&\mathcal{L}_\evf(\dd v+\omega\cdot\evf)\\
  =&\dd v+\omega\cdot\evf+(\mathcal{L}_\evf\omega)\cdot\evf\\
  =&\mathcal{L}_\evf\tilde{\theta}+S(\evf,\theta)\evf,
 \end{align*}
 from which we obtain:
 \[
  \mathcal{L}_\evf(\mathcal{L}_\evf-\Id)\tilde{\theta}=S(\evf,\theta)\evf=\mathcal{L}_\evf(\mathcal{L}_\evf-\Id)\theta.
 \]
 
 That is, it holds:
 \[
  (\mathcal{L}_\evf-\Id)\tilde{\theta}=(\mathcal{L}_\evf-\Id)\theta,
 \]
 or equivalently, 
 \[
  \mathcal{L}_\evf(\tilde{\theta}-\theta)=\tilde{\theta}-\theta,
 \]
 which by Lemma \ref{Lem:FixedPointsOfLieE} implies,
 \[
  \tilde{\theta}-\theta\equiv(\tilde{\theta}-\theta)(0)=0.
 \]
 
 To \emph{ii)}: In light of \emph{i)} it holds:
 \begin{align*}
  \mathcal{L}_\evf(\dd\theta+\omega\wedge\theta)=&\dd\mathcal{L}_\evf\theta+\mathcal{L}_\evf\omega\wedge\theta+\omega\wedge\mathcal{L}_\evf\theta\\
  =&\dd(\omega\cdot\evf)+S(\evf,\theta)\wedge\theta+\omega\wedge\dd v+\omega\wedge(\omega\cdot\evf)\\
  =&(\dd\omega)\cdot\evf-\omega\wedge\dd v+S(\evf,\theta)\wedge\theta+\omega\wedge\dd v+(\omega\wedge\omega)\cdot\evf\\
  =&(\dd\omega+\omega\wedge\omega)\cdot\evf+S(\evf,\theta)\wedge\theta\\
  =&(\dd\omega+\omega\wedge\omega-S(\theta,\theta))\cdot\evf,
 \end{align*}
 where the last equality follows from the fact that $S$ takes values in $K(\frak{g})$, and so it holds:\\ $S(\evf,\theta)\wedge\theta=-S(\theta,\theta)\cdot\evf$.
 \end{proof}
 
 We now have all the ingredients to prove Theorem \ref{Thm:MainThm}:
 
 \begin{theorem}\label{Thm:MainThm}
  Let $U\subset V$ a star-shaped around $0$ open subset, let $S\colon U\longto K(\frak{g})$ be a real-analytic map, $\theta$ the analytic solution to the singular initial value problem
 \[
  \begin{cases}
   \mathcal{L}_\evf(\mathcal{L}_\evf-\Id)\theta=S(\evf,\theta)\evf,\\
   \theta_0=\Id,\,\dd\theta_0=0,
  \end{cases}
 \]
and suppose $\omega\coloneqq I(S(\evf,\theta))$ satisfies the consistency relation $\dd\omega+\omega\wedge\omega=S(\theta,\theta)$. Then there exists a unique torsion-free analytic connection $\nabla$ on a sufficiently small open neighborhood $U'\subset U$ of $0$ such that for all $v\in U'$:
  \begin{equation}\label{Eq:S=PR}
   S_v=P_{1,0}\cdot \Rnabla_{\gamma_v(1)},
  \end{equation}
  where $P_{1,0}$ denotes the inverse of the parallel translation map along the radial geodesic $\gamma_v(t)\coloneqq tv$, and $\Rnabla$ denotes the curvature tensor of the connection $\nabla$.
 \end{theorem}
 \begin{proof}
  Let $U'$ be the open neighborhood of $0$ in which $\theta$ is real-analytic. Let $g\colon U'\longto G$ be the map $g\coloneqq\Id+I(\omega\cdot\evf)$. Set $\Gamma\in\Omega^1(U',\frak{g})$ as 
  \[
   \Gamma\coloneqq\Ad(g^{-1})\circ\omega+g^*\mu_G=\Ad(g^{-1})\circ\omega+g^{-1}\dd g.
  \]
  
  It thus holds:
  \begin{align*}
   \Gamma\wedge\dd v=&g^{-1}(\omega g+\dd g)\wedge\dd v\\
   =&g^{-1}(\omega\wedge g\dd v+\dd g\wedge\dd v)\\
   =&g^{-1}(\omega\wedge\theta+\dd\theta)\\
   =&0,
  \end{align*}
  where the last equality follows from the fact that $\dd\omega+\omega\wedge\omega=S(\theta,\theta)$ and Proposition \ref{Prop:ThetaProperties}.
  
  From this we thus conclude, that the analytic map $\Gamma\colon U'\longto V^*\otimes\frak{g}$ actually takes values in $\Sym^2V^*\otimes V$, and so it uniquely defines a torsion-free connection $\nabla$ by the formula:
  \[
   \nabla_Xs\coloneqq\dd s(X)+\Gamma(X)s,
  \]
  for every smooth sections $X,s\colon U'\longto V$.
  
  The curvature of $\nabla$ is then given by the formula
  \[
   \Rnabla=\dd\Gamma+\Gamma\wedge\Gamma.
  \]
  
  A direct computation shows
  \[
   \dd\Gamma=\Ad(g^{-1})\dd\omega-\Ad(g^{-1})\omega\wedge g^*\mu_G-g^*\mu_G\wedge\Ad(g^{-1})\omega
  \]
  and so it holds
  \[
   \Rnabla=\Ad(g^{-1})(\dd\omega+\omega\wedge\omega)+g^*\overbrace{(\dd\mu_G+\mu_G\wedge\mu_G)}^{=0}=\Ad(g^{-1})(\dd\omega+\omega\wedge\omega).
  \]
  
  Since it also holds that $\dd\omega+\omega\wedge\omega=S(\theta,\theta)=S(g,g)$, it follows from \eqref{Eq:RepRhom}:
  \[
   \Rnabla=\Ad(g^{-1})S(g,g)=\rho_2(g^{-1})S,
  \]
  which at once shows that $R^\nabla$ indeed takes values in $K(\frak{g})$ and,
  \[
   S=\rho_2(g)\Rnabla.
  \]
  
  Finally, that for the analytic map $g^{-1}\colon U'\longto G$, holds $g(v)^{-1}=P_{\gamma_v}$ follows from the mere definition of $\nabla$. 
  
  In sum it thus holds for every $v\in U'$:
  \[
   S_v=\rho_2(P_{\gamma_v}^{-1})\Rnabla_v=P_{1,0}\cdot \Rnabla_{\gamma_v(1)}.
  \]
 \end{proof}
 
 In the context of Theorem \ref{Thm:MainThm}, we obtain the following:
 \begin{corollary}
  Let $S$, $\theta=g\dd v$, $\omega$, and $U'$ be as in Theorem \ref{Thm:MainThm}. It thus holds, the map
  \[
   (\dd S+(\rho_2)_*(\omega)\wedge S)(g^{-1})\colon U'\longto V^*\otimes K(\frak{g}),
  \]
  takes values in the subspace $K^1(\frak{g})$.
 \end{corollary}
 \begin{proof}
  From Theorem \ref{Thm:MainThm}, it immediately follows that $\omega=\sigma^*\omega^\nabla$, where $\omega^\nabla$ is the connection form associated to the torsion-free connection $\nabla$ from the Theorem, and $\sigma=(\Id_{U'},g^{-1})$ denotes its associated exponential framing.
  
  From Proposition \ref{Prop:CurvatureAndGauges}, and the fact that $S=\rho_2(g)R^\nabla$ it thus follows that
  \[
   (\dd S+(\rho_2)_*(\omega)\wedge S)(g^{-1})=\rho_3(g)\nabla\Rnabla,
  \]
  from which the result follows, since $\Rnabla$ satisfies the second Bianchi identity, which amounts to the fact that $\nabla \Rnabla$ takes values in $K^1(\frak{g})$.
 \end{proof}
 
 \section{Holonomy of Torsion-free Connections}\label{Sec:Holonomy}
 
 We are now concerning ourselves with some relevant consequences of Theorem \ref{Thm:MainThm} in the direction of Holonomy Theory. In particular, the holonomy of the torsion-free connection induced by the analytic map $S$ in Theorem \ref{Thm:MainThm} can conveniently be characterized by it. Concretely, one has the following:
 
 \begin{theorem}\label{Thm:Thm2}
  Let $S\colon U\longto K(\frg)$, $\theta\in\Omega^1(U,V)$, $\omega\in\Omega^1(U,\frg)$ as in Theorem \ref{Thm:MainThm}. Let $\nabla$ be the torsion-free connection on $TU'$ which satisfies \eqref{Eq:S=PR}. It holds:
  \begin{equation}\label{Eq:holOfNablaS}
   \hol_0(\nabla)=\gen\qty{S_v(x,y)\;\vert\;v\in U',\,x,y\in V}.
  \end{equation}
 \end{theorem}
 \begin{proof}
 We denote the right-hand side of \eqref{Eq:holOfNablaS} by $\frh$. Because of \eqref{Eq:S=PR} and the Ambrose-Singer Theorem \ref{Prop:Ambrose-Singer}, we immediately obtain $\frh\subset\hol_0(\nabla)$.
 
 Let $H$ be the connected Lie subgroup of $G$ which has $\frh$ as its Lie algebra, and set
 \[
  F_H\coloneqq\bigsqcup_{v\in U'}\qty{\sigma(v)g\;\vert\;v\in U',\,g\in H},
 \]
 where $\sigma\colon U'\longto U'\times G$ denotes the exponential framing associated to $\nabla$. That is,\\ $\sigma(v)=(v,h(v))\coloneqq(v,P_{\gamma_v})$. Notice that $F_H\longto U'$ naturally admits the structure of an $H$-principal bundle, and as such, it is isomorphic to the $H$-bundle $U'\times H\longto U'$.
 
 Let $\gamma\colon[0,1]\longto F_H$, and set $\gamma(0)=\sigma(v)g$. Thus there exists curves $\gamma_1\colon[0,1]\longto U'$ with $\gamma_1(0)=v$, $\gamma_2\colon[0,1]\longto H$ with $\gamma_2(0)=g$ such that $\gamma(t)=\sigma(\gamma_1(t))\gamma_2(t)$.
 
 From this we thus get:
 \begin{align*}
  \omega^\nabla_{\sigma(v)g}(\gamma'(0))=&\Ad(g^{-1})\circ\omega_{\sigma(v)}^\nabla(\dd R_{g^{-1}}\gamma'(0))\\
  =&\Ad(g^{-1})\circ\omega_{\sigma(v)}^\nabla\qty(\eval{\dv{}{t}}_0\sigma(\gamma_1(t))\gamma_2(t)g^{-1})\\
  =&\Ad(g^{-1})\circ\omega_{\sigma(v)}^\nabla\qty((\gamma_1'(0),h(v)\gamma_2'(0)g^{-1}+\dd_vh(\gamma_1'(0))))\\
  =&\Ad(g^{-1})(\gamma_2'(0)g^{-1}+\omega_v(\gamma_1'(0)))\\
  =&g^{-1}\gamma_2'(0)+\Ad(g^{-1})\omega_v(\gamma_1'(0))\\
  =&\dd L_{g^{-1}}(\gamma_2'(0))+\Ad(g^{-1})\omega_v(\gamma_1'(0)).
 \end{align*}
 
 Now, because of the fact that $\omega=I(S(\evf,\theta))$ is analytic on $U'$, with power series given by \eqref{Eq:FormalPowerSeriesI}, we then conclude that $\omega$ takes in fact values in $\frh$, which in turn implies
 \[
  \omega_{\sigma(v)g}^\nabla(\gamma'(0))=\dd L_{g^{-1}}(\gamma_2'(0))+\Ad(g^{-1})\omega_v(\gamma_1'(0))\in\frh.
 \] 
 
 In other words, we have just proved that the restricted connection form $\omega^\nabla\vert_{TF_H}$ takes values in $\frh$, whichs amounts to the fact that $(F(U'),\omega^\nabla)$ is reducible to $(F_H,\omega^\nabla\vert_{TF_H})$ \cite[Chapter~4, Satz~4.2]{baum2009eichfeldtheorie}, and so, by Proposition \ref{Prop:SmallestReduction} we conclude in particular, that for every $u\in F_H$,
 \[
  P^\nabla(u)\subset F_H=U'\times H.
 \]
 
 By setting $u_0\coloneqq\sigma(0)=(0,\Id)\in F_H$, we obtain:
 \[
  P^\nabla(u_0)=U'\times\Hol_{u_0}(\omega^\nabla)=U'\times\Hol_0(\nabla)\subset U'\times H.
 \]
 
 That is, $\Hol_0(\nabla)\subset H$, which together with the observation that $\frh\subset\hol_0(\nabla)$ yields 
 \[
  \frh=\hol_0(\nabla).
 \]
 \end{proof}
 
 In light of Theorem \ref{Thm:Thm2} we obtain at once the following useful
 \begin{corollary}\label{Cor:HolonomyInSubalg}
  In the situation of Theorem \ref{Thm:MainThm}. The holonomy algebra of the torsion-free connection induced by the analytic map $S$ is contained in the Lie subalgebra $\frh\subset\frg$ if, and only if, the map $S$ takes values in the subspace $K(\frh)$.
 \end{corollary}
 
 \bibliographystyle{alpha}
 \bibliography{CurvOfTorsionfreeConn}
\end{document}